\documentclass[12pt]{amsart}   
\usepackage{amsmath,amsthm,amssymb}
\newdimen\slantmathcorr
\def\oversl#1{
\setbox0=\hbox{$#1$}
\slantmathcorr=\wd0
\hskip 0.2\slantmathcorr \overline{\hbox to 0.8\wd0{
\vphantom{\hbox{$#1$}}}}
\hskip-\wd0\hbox{$#1$}
}

\newtheorem{theorem}{Theorem}[section]
\newtheorem{Lemma}[theorem]{Lemma}

\newtheorem{Corollary}[theorem]{Corollary}
\newtheorem{Definition}[theorem]{Definition}
\newtheorem{Example}[theorem]{Example}

\begin{document} 
\title{Measure expansivity and specification for Pointwise dynamics}     
\author{Pramod Das, Abdul Gaffar Khan, Tarun Das}                
\maketitle          
\begin{abstract}
We introduce pointwise measure expansivity for bi-measurable maps. We show through examples that this notion is weaker than measure expansivity. In spite of this fact, we show that many results for measure expansive systems hold true for pointwise systems as well. Then, we study the concept of mixing, specification and chaos at a point in the phase space of a continuous map. We show that mixing at a shadowable point is not sufficient for it to be a specification point, but mixing of the map force a shadowable point to be a specification point. We prove that periodic specification points are Devaney chaotic point. Finally, we show that existence of two distinct specification points is sufficient for a map to have positive Bowen entropy.      
\end{abstract} 
\medskip

\medskip

\medskip

\section{Introduction}
\medskip

One of the extensively studied topological dynamical notions for homeomorphisms is widely known as expansivity which was introduced \cite{U} by Utz in the middle of the twentieth century. This notion says that we can choose a fixed $\delta>0$ such that for each $x\in X$, the orbit of any point get separated from the orbit of $x$ by the constant $\delta$. In \cite{R1970}, Reddy constructed homeomorphisms which does not fulfil this requirement but satisfy weaker condition. He called such homeomorphisms as pointwise expansive homeomorphisms. In particular, for given $x\in X$, he allowed to choose different $\delta_x>0$ by which other orbits get separated from the orbit of $x$. He observed that many results which are true for expansive homeomorphisms are also true for pointwise expansive homeomorphisms. 
\medskip

In recent years, Morales has extended \cite{MS} the notion of expansivity for homeomorphisms to Borel measures. He has called a Borel measure expansive if we can choose a fixed $\delta>0$ such that for each $x\in X$, the orbit of all points except possibly a set of measure zero get separated from the orbit of $x$ by the constant $\delta$. If this condition is satisfied by any non-atomic Borel measure on the phase space of a homeomorphism, then it is called a measure expansive homeomorphism. In this paper, we introduce pointwise expansivity for Borel measures by allowing different $\delta_x>0$ for given $x\in X$ such that the orbit of all points except possibly a set of measure zero get separated from the orbit of $x$ by the constant $\delta_x$. Further, if this condition is satisfied by any non-atomic Borel measure on the phase space of a bi-measurable map, then the map is called pointwise measure   
\par\noindent\rule{\textwidth}{0.4pt}
Department of Mathematics, Faculty of Mathematical Sciences, 
\\
University of Delhi, New Delhi-110007
\\ 
tarukd@gmail.com (Tarun Das), pramod.math.ju@gmail.com, gaffarkhan18@gmail.com 

expansive. Although we have been able to realize by examples that this notion is weaker than measure expansivity, many results of measure expansive systems hold true for pointwise systems as well. 
\medskip

It is always fascinating to find easier way to understand interesting behaviours of a dynamical system. For example, positive entropy of a system indicates chaotic behaviour, but it is one of the most difficult job to find entropy of a system. Therefore, mathematicians started working to understand whether presence of some dynamical notion in the system implies positive entropy. In \cite{AT}, it was proved that every homeomorphism on a compact metric space possessing specification property has positive topological entropy. In \cite{AC}, authors have shown that $P$-chaotic map on a continuum has positive entropy. In \cite{M}, Moothathu has proved that continuous self-map on a compact metric space possessing shadowing property having either a non-minimal recurrent point or a sensitive minimal subsystem, has positive topological entropy. 
\medskip

In spite of the development of such significant literature on global behaviour of a dynamical system, mathematicians started working in finding those local behaviours which provides information about the global behaviour. In \cite{XZ}, authors have introduced the concept of entropy points and proved that existence of such a point implies positive entropy. Very recently, Kawaguchi has introduced \cite{KQS} a weaker version of shadowable points \cite{MSP} called $e$-shadowable points and in a subsequent paper \cite{KPSP} has shown that existence of certain kind of $e$-shadowable points implies positive entropy. In a preprint \cite{AR}, Arbieto and Rego have shown that continuous map on a compact metric space having non-periodic shadowable, positive uniformly countable expansive, non-wandering point whose every neighborhood is uncountable has positive topological entropy. Here, we show that existence of two distinct specification points implies positive Bowen entropy. Besides that we have shown that every shadowable point of a topologically mixing continuous map is a specification point and that every periodic specification point of a continuous surjective map is a Devaney chaotic point.
\medskip

The notion of specification is one of the most important and extensively studied variant of shadowing. In shadowing one traces an approximate orbit but specification guarantees simultaneous tracing of finite number of finite pieces of orbits by one periodic orbit. Such kind of tracing by a periodic orbit helped Bowen to study \cite{B} the distribution of periodic points in the phase space of a system with specification property. The periodicity of the tracing point made this notion to be popularly known as periodic specification property among many other variants (\cite{KO},\cite{S}).  

\medskip

This paper is presented as follows. In section 2, the essential preliminaries are supplied to make the paper self-contained. In section 3, we have done a qualitative study of pointwise measure expansive systems and have constructed examples of such systems which are not measure expansive. In particular, we show that the set of points with converging semiorbits under a pointwise measure expansive homeomorphism on a separable metric space has measure zero with respect to any non-atomic Borel measure and every uniform equivalence of a separable metric space is aperiodic with respect to any pointwise expansive measure. We also show that the set of points positively asymptotic to a point and negatively asymptotic to another point under a pointwise measure expansive homeomorphism of a separable metric space has measure zero with respect to any non-atomic Borel measure. We further show that the set of heteroclinic points of a pointwise measure expansive uniform equivalence of a separable metric space has measure zero with respect to any non-atomic Borel measure, using the fact that the set of periodic points of such system is at most countable. We then show that the number of stable classes of a measurable map admitting positively pointwise expansive outer regular measure is uncountable. We further show that a bi-measurable map with canonical coordinates admitting strictly positive positively pointwise expansive measure has no point which is a sink. In section 4, we study periodic specification, topological mixing and Devaney chaos at a point. We prove that every shadowable point of a topologically mixing continuous map is a point with specification and that every periodic specification point of a continuous surjective map is a Devaney chaotic point. Further, we show that every uniformly continuous surjective map having two distinct specification points has positive Bowen entropy.

\section{PRELIMINARIES} 
Throughout the paper, $X$ denotes any metric space (unless otherwise stated). A point $x\in X$ is called an atom for a measure $\mu$ if $\mu(\lbrace x\rbrace)>0$. A measure $\mu$ on $X$ is said to be non-atomic if it has no atom. We call $X$ non-atomic if there exists a non-atomic Borel measure on it. A Borel measure is called strictly positive if every open set has positive measure.  
\medskip

A bi-measurable map $f$ on $X$ is called expansive \cite{U} if there is $\delta>0$ such that for any pair of distinct points $x,y\in X$, we have $d(f^n(x),f^n(y))>\delta$ for some $n\in\mathbb{Z}$. In other words, $\Gamma_{\delta}(x)=\lbrace x\rbrace$ for all $x\in X$, where $\Gamma_{\delta}(x)=\lbrace y\in X\mid d(f^n(x),f^n(y))\leq\delta$ for all $n\in\mathbb{Z}\rbrace$. A bi-measurable map $f$ on $X$ is measure-expansive \cite{MS} if for any non-atomic Borel measure $\mu$ on $X$, we have $\mu(\Gamma_{\delta}(x))=0$ for all $x\in X$. A bi-measurable map $f$ on $X$ is called pointwise expansive \cite{R1970} if for each $x\in X$ there is $\delta_{x}>0$ (depending on $x$) such that $\Gamma_{\delta_x}(x)=\lbrace x\rbrace$. 
\medskip

A subset $A\subset X$ is said to be $f$-invariant if $f^{-1}(A)= A$ and $f$ is said to be $\mu$-invariant if $\mu(f^{-1}(A)) = \mu(A)$ for every measurable subset $A\subset X$. A non-atomic Borel measure $\mu$ is said to be pointwise expansive for a bi-measurable map $f$ through a subset $\mathbb{H}$ of $\mathbb{Z}$, if for each $x\in X$ there is $\delta_{x} > 0$ such that $\mu(\Gamma^{\mathbb{H}}_{\delta_{x}}(x)) = 0$, where $\Gamma^{\mathbb{H}}_{\delta_{x}}(x) = \lbrace y\in X : d(f^{n}(x), f^{n}(y))\leq \delta_{x}$ for all $n\in \mathbb{H}\rbrace$. 
\medskip

Let $f:X\rightarrow X$ be a bi-measurable map and let $\mu$ be a Borel measure on $X$. A finite open cover $\Lambda= \lbrace A_{1}, A_{2}, . . ., A_{n} \rbrace$ of a topological space $X$ is said to be a $\mu$-generator at $x$ for $f$ if for any bi-sequence $\lbrace U_{n} \rbrace_{n \in \mathbb{Z}}$ of members of $\Lambda$ satisfying $x\in f^{-n}(\oversl{U_{n}})$ for all $n\in \mathbb{Z}$, $\mu(\cap_{n \in  \mathbb{Z}} f^{-n}(\oversl{U_{n}})) = 0$. We say that $f$ has pointwise $\mu$-generator if for each $x\in X$ there exists a $\mu$-generator at $x$ for $f$.     
\medskip
 
The $\omega$-limit set of a point $x\in X$ under a measurable map $f$ is given by $\omega(f,x)
\\
=\lbrace y\in X\mid$ lim$_{k\to\infty} d(f^{n_k}(x),y)=0$ for some strictly increasing sequence $(n_k)\rbrace$. 

The $\alpha$-limit set of a point $x\in X$ under a bi-measurable map $f$ is given by $\alpha(f,x)
\\
=\lbrace y\in X\mid$ lim$_{k\to\infty} d(f^{n_k}(x),y)=0$ for some strictly decreasing sequence $(n_k)\rbrace$.  
\medskip

We say that a point $x\in X$ has converging semiorbits under a bi-measurable map $f$ if both $\alpha(f,x)$ and $\omega(f,x)$ consist of single point. The set of such points under $f$ is denoted by $A(f)$.  
\medskip

A map $f$ is said to be aperiodic with respect to a measure $\mu$ if $\mu(B)=0$, whenever $B\subset X$ is such that there is $n\in\mathbb{N}^+$ satisfying $f^n(x)=x$ for all $x\in B$. 
\medskip

A point $y\in X$ is said to be positively asymptotic to $x\in X$ under a measurable map $f$ if for every $\epsilon>0$ there is a positive integer $N$ such that $d(f^n(x),f^n(y))<\epsilon$ for all $n\geq N$. The point $y\in X$ is said to be negatively asymptotic to $x\in X$ under a bi-measurable map $f$ if for every $\epsilon>0$ there is a positive integer $M$ such that $d(f^n(x),f^n(y))<\epsilon$ for all $n\leq -M$. The set of points positively asymptotic to $p$ and negatively asymptotic to $q$ under a bi-measurable map $f$ is denoted by $A_f(p,q)$. The point $y\in X$ is said to be doubly asymptotic to $p\in X$ under a bi-measurable map $f$ if $y$ is both positively and negatively asymptotic to $p$.   
\medskip

The stable set of a point $x\in X$ under a measurable map $f$ is given by  
\\
$W^s(x)=\lbrace y\in X\mid$ for all $\epsilon>0$, there is $n\in\mathbb{N}$ such that $d(f^i(x),f^i(y))\leq\epsilon$ for all $i\geq n\rbrace$. 

For $\delta>0$, the local stable set of a point $x\in X$ under a measurable map $f$ is given by
\\
$W^s(x,\delta)=\lbrace y\in X\mid d(f^i(x),f^i(y))\leq\delta$ for all $i\geq 0\rbrace$.   

The unstable set of a point $x\in X$ under a bi-measurable map $f$ is given by  
\\
$W^u(x)=\lbrace y\in X\mid$ for all $\epsilon>0$, there is $n\in\mathbb{N}$ such that $d(f^{i}(x),f^{i}(y))\leq\epsilon$ for all $i\leq -n\rbrace$. 

For $\delta>0$, the local unstable set of a point $x$ under a bi-measurable map is given by 
\\
$W^u(x,\delta)=\lbrace y\in X\mid d(f^i(x),f^i(y))\leq\delta$ for all $i\leq 0\rbrace$.    
\medskip

A point $x$ is said to be heteroclinic for a bi-measurable map $f$ if $x\in W^s(O(p))\cap W^u(O(q))$. A point $x$ is called a sink for a bi-measurable map $f$ if $W^u(x,\delta)=\lbrace x\rbrace$ for some $\delta>0$. The map $f$  is said to have canonical coordinates if for every $\epsilon>0$ there is $\delta>0$ such that $d(x,y)<\delta$ implies $W^s(x,\epsilon)\cap W^u(y,\epsilon)\neq \phi$.  
\medskip

For a continuous map $f$ on a metric space $(X,d)$, recall the following topological dynamical notions.  
\\
(i) $f$ is sensitive if there exists $\delta > 0$ such that for every point $x\in X$ and every neighborhood $U$ of $X$ there exists $y\in U$ such that $d(f^{n}(x), f^{n}(y)) > \delta$ for some $n\geq 1$.    
\\
(ii) $f$ is said to be topologically transitive if for every pair of non-empty open sets $U$ and $V$ there exists $n\in \mathbb{N}$ such that $f^{n}(U)\cap V\neq \phi$. 
\\
(iii) $f$ is said to be topologically mixing if for every pair of non-empty open sets $U$ and $V$ there exists $N\in \mathbb{N}$ such that $f^{n}(U)\cap V\neq \phi$ for all $n\geq N$. 
\\
(iv) $f$ is said to be Devaney chaotic if it is topologically transitive, has dense set of periodic points in $X$ and is sensitive.    
\\
(v) $f$ is said to have specification (resp. periodic specification) property if for every $\epsilon > 0$ there exists a positive integer $M(\epsilon)$ such that for any finite sequence $x_{1}, x_{2}, . . ., x_{k}$ in $X$, any integers $0\leq a_{1}\leq b_{1} < a_{2}\leq b_{2} < . . .< a_{k}\leq b_{k}$ with $a_{j} - b_{j-1} \geq M(\epsilon)$ for all $1\leq j \leq k$, there exists $y\in X$ (resp. $y\in Per(f)$) such that $d(f^{i}(y), f^{i}(x_{j})) < \epsilon$ for all $a_{j}\leq i\leq b_{j}$ and all $1\leq j\leq k$.
\\
(v) For $\delta > 0$, any sequence $\lbrace x_{i}\rbrace_{i\in \mathbb{N}}$ is said to be a $\delta$-pseudo orbit for $f$ if $d(f(x_{i}), x_{i+1}) < \delta$ for all $i\in \mathbb{N}$. For $\epsilon > 0$, any sequence $\lbrace x_{i}\rbrace_{i\in \mathbb{N}}$ is said to be $\epsilon$-traced if there exists a point $x\in X$ such that $d(f^{i}(x), x_{i}) < \epsilon$ for all $i\in \mathbb{N}$. The map $f$ is said to have shadowing if for every $\epsilon>0$ there exists $\delta>0$ such that every $\delta$-pseudo orbit is $\epsilon$-traced by some point in $X$.    
\medskip
 
Let $f$ be a uniformly continuous surjective map on a metric space $(X,d)$. Let $n\in \mathbb{Z}^{+}$, $\epsilon > 0$ and $K\in \mathcal{K}(X)$, where  $\mathcal{K}(X)$ denotes the set of all compact subsets of $X$. Set $d_{n}(x, y) = $max$_{0\leq i\leq (n-1)}d(f^{i}(x), f^{i}(y))$. A subset $F$ of $K$ is said to $(n, \epsilon)$-separated subset of $K$ if for every pair of distinct points $x, y\in E$, $d_{n}(x, y) > \epsilon$. The number $s_{n}(\epsilon, K)$ denotes the largest cardinality of $(n, \epsilon)$-separated subset of $K$. Set $s(\epsilon, K) = $lim sup$_{n\rightarrow \infty}\frac{1}{n}log(s_{n}(\epsilon, K))$. The entropy of $f$ is defined as $h(f) = sup_{K\in \mathcal{K}(X)} $lim$_{\epsilon \rightarrow 0} s(\epsilon, K)$. For deeper understanding of this notion of entropy due to Bowen one may refer to \cite{W}.     
\section{Pointwise measure expansivity} 
In this section, we study the following notions which are seen to be pointwise versions of measure expansivity.  
\\
(i) Let $X$ be a metric space and let $f$ be a bi-measurable map on $X$. Then, a non-atomic Borel measure $\mu$ on $X$ is said to be pointwise (resp. positively pointwise) expansive for $f$, if for each $x\in X$ there exists $\delta_x>0$ such that $\mu(\Gamma_{\delta_x}(x))=0$ (resp. $\mu(\Phi_{\delta_x}(x))=0$), where $\Phi_{\delta_x}(x)=\lbrace y\in X\mid d(f^n(x),f^n(y))\leq \delta_x$ for all $n\in\mathbb{N}\rbrace$.   
\\
(ii) Let $X$ be a non-atomic metric space and let $f$ be a bi-measurable map on $X$. Then, $f$ is said to be pointwise measure expansive if for each $x\in X$ there is $\delta_x>0$ (depending on $x$) such that $\mu(\Gamma_{\delta_x}(x))=0$ for any non-atomic Borel measure $\mu$. 
\\
(iii) Let $X$ be a metric space and let $f$ be a bi-measurable map on $X$. Then, $f$ is said to be strongly pointwise measure expansive if for each $x\in X$ there is $\delta_x>0$ (depending on $x$) such that $\mu(\Gamma_{\delta_x}(x))=\mu(\lbrace x\rbrace)$ for any Borel measure $\mu$ on $X$.
\\
(iv) Let $X$ be a metric space and let $f$ be a bi-measurable map on $X$. Then, $f$ is said to be pointwise $N$-expansive if for each $x\in X$, there exists $\delta_x>0$ such that $\mid\Gamma_{\delta_x}(x)\mid\leq N$.
\begin{theorem}
Let $f$ be a uniform equivalence on a metric space $X$ and let $\mu$ be a non-atomic Borel measure on $X$. Then, $\mu$ is pointwise expansive for $f$ if and only if $\mu$ is pointwise expansive for $f$ through $\mathbb{H}$, for every non-trivial subgroup of $\mathbb{Z}$. 
\label{3.1}
\end{theorem}
\begin{proof}
Let $\mu$ be pointwise expansive at $x\in X$ with pointwise expansive constant $\delta_{x}$ and let $\mathbb{H}$ be a non-trivial subgroup of $\mathbb{Z}$. Then, we can choose positive integers $\lbrace 1, 2,$...,$k\rbrace$ such that $\mathbb{Z} =\bigcup \lbrace i+ \mathbb{H}\mid 1\leq i \leq k\rbrace$. By uniform continuity of $f$ there is $\epsilon_{x} > 0$ such that $d(a, b)\leq \epsilon_{x}$ implies that $d(f^{i}(a), f^{i}(b))\leq \delta_{x}$ for all $1\leq i\leq k$. Since for each $n\in\mathbb{Z}$ there is $i$ such that $ n = i + \mathbb{H}$, we have that $\Gamma^{\mathbb{H}}_{\epsilon_{x}}(x)\subset \Gamma_{\delta_{x}}(x)$ and hence $\mu(\Gamma^{\mathbb{H}}_{\epsilon_{x}}(x))= 0$. The converse follows from the definition. 
\end{proof}
\begin{Corollary}
Let $f$ be a uniform equivalence on a metric space $X$ and let $\mu$ be a non-atomic Borel measure on $X$. Then, $\mu$ is pointwise expansive for $f$ if and only if it is pointwise expansive for $f^m$ for any $m\in\mathbb{Z}\setminus \lbrace 0\rbrace$.    
\label{3.2}
\end{Corollary}
\begin{theorem}
Let $f$ be a bi-measurable map on a compact metric space $X$ and let $\mu$ be a non-atomic Borel measure on $X$. Then, $\mu$ is pointwise expansive for $f$ if and only if $f$ has a pointwise $\mu$-generator.    
\end{theorem}
\begin{proof} Suppose that $\mu$ is pointwise expansive for $f$ and $\delta_{x}$ be a pointwise expansivity constant for $\mu$ at $x$. Choose an open cover $\Lambda$ of $X$ containing open-balls of radius $\delta_{x}$. Choose any bi-sequence $\lbrace U_{n} \rbrace_{n \in \mathbb{Z}}$ of members of $\Lambda$ satisfying $x\in f^{-n}(\oversl{U_{n}})$ for all $n\in \mathbb{Z}$. Clearly, $\cap_{n \in  \mathbb{Z}} f^{-n}(\oversl{U_{n}}) \subset \Gamma_{\delta_{x}}(x)$ which implies $\mu(\cap_{n \in  \mathbb{Z}} f^{-n}(\oversl{U_{n}})) \leq \mu(\Gamma_{\delta_{x}}(x)) = 0$. Thus, $\Lambda$ is a pointwise $\mu$-generator for $f$.  
\medskip

Conversely, suppose that $\Lambda$ is a pointwise $\mu$-generator for $f$ at $x$ and let $\delta_x > 0$ be its Lebesgue number. Choose a sequence $\lbrace U_{n} \rbrace_{n \in \mathbb{Z}}$ such that every closed $\delta_x$-ball around $f^{n}(x)$ is contained in $\oversl{U_{n}}$ for all $n\in \mathbb{Z}$. Thus, we have $\Gamma_{\delta_{x}}(x)\subset \cap_{n \in  \mathbb{Z}} f^{-n}(\oversl{U_{n}}) $ which implies $ \mu(\Gamma_{\delta_{x}}(x)) \leq \mu(\cap_{n \in  \mathbb{Z}} f^{-n}(\oversl{U_{n}})) = 0$, Thus, $\mu$ is pointwise expansive for $f$.  
\end{proof}

For $x,y\in X$, $n,m\in\mathbb{N}^+$, we define 
\begin{center} 
$A(x,y,n,m)=\lbrace z\in X\mid$ max$\lbrace d(f^{-i}(z),x),d(f^i(z),y)\rbrace\leq\frac{1}{n}$ for all $i\geq m\rbrace$.   
\end{center} 
\medskip

\begin{Lemma}
Let $X$ be a separable metric space and let $f:X\rightarrow X$ be a bi-measurable map, then there exists a sequence $\lbrace x_i\rbrace_{i\in\mathbb{N}}$ such that 
\begin{center}
$A(f)\subset\bigcap_{n\in\mathbb{N}^+}\bigcup_{k,k',m\in\mathbb{N}^+} A(x_k,x_{k'},n,m)$. 
\end{center}
\label{3.4}
\end{Lemma}
\begin{proof}
If $z\in A(f)$, then $\alpha(f,z)$ and $\omega(f,z)$ reduce to single points $x$ and $y$ respectively. Then, for each $n\in\mathbb{N}^+$, there is $m\in \mathbb{N}^+$ such that $d(f^{-i}(z),x)\leq\frac{1}{2n}$ and $d(f^i(z),y)\leq\frac{1}{2n}$ for all $i\geq m$. If $\lbrace x_i\rbrace_{i\in\mathbb{N}}$ is dense in $X$, there are $k,k'\in\mathbb{N}^+$ such that $d(x_k,x)\leq\frac{1}{2n}$ and $d(x_{k'},y)\leq\frac{1}{2n}$. Therefore, max$\lbrace d(f^{-i} (z),x_k),d(f^i(z),x_{k'})\leq\frac{1}{n}$ for all $i\geq m$. This completes the proof. 
\end{proof}

\begin{Lemma}
Let $\mu$ be a Borel measure on $X$. Then, for every measurable Lindel{\"o}f subset $K$ with $\mu(K)>0$ there are $z\in K$ and $\delta_0>0$ such that $\mu(K\cap B[z,\delta])>0$ for all $0<\delta<\delta_0$.  
\label{3.5}
\end{Lemma}
\begin{proof}
Otherwise, for every $z\in K$ there is $0<\delta_z<\delta_0$ such that $\mu(K\cap B[z,\delta_z])=0$. Since $K$ is Lindel{\"o}f, the open cover $\lbrace K\cap B[z,\delta_z]\mid z\in K\rbrace$ of $K$ admits a countable sub-cover, i.e., there is a sequence $\lbrace z_l\rbrace_{l\in\mathbb{N}}$ in $K$ satisfying $K=\bigcup_{l\in\mathbb{N}}(K\cap B[z_l,\delta_{z_l}])$. So, $\mu(K)=\sum_{l\in\mathbb{N}}\mu(K\cap B[z_l,\delta_{z_l}])=0$, a contradiction. 
\end{proof}

\begin{theorem}
The set $A(f)$ of a bi-measurable map $f$ of a separable metric space $X$, has measure zero with respect to any pointwise expansive measure.
\label{3.6}
\end{theorem}
\begin{proof}
Let $\mu$ be a pointwise expansive outer regular measure for $f$. By contradiction, suppose there is $A\subset A(f)$ such that $\mu(A)>0$. For each $n\in\mathbb{N}^+$, let $A(n)$ be the set of points $a\in A$ such that $\frac{1}{n}\leq\delta_a$. Since $A=\bigcup_{n\in\mathbb{N}^+} A(n)$, we have $\mu(A(M))>0$ for some $M\in\mathbb{N}^+$. By Lemma \ref{3.4}, there is a sequence $x_k\in X$ such that $A(M)\subset \bigcap_{n\in\mathbb{N}^+}\bigcup_{k,k',m\in\mathbb{N}^{+}} A(x_k,x_{k'},n,m)$. Therefore, we have $A(M)\subset \bigcup_{k,k',m\in\mathbb{N}^{+}} A(x_k,x_{k'},n,m)$ for all $n\in\mathbb{N}^+$. Thus, we can choose $k,k',m\in\mathbb{N}^+$ such that $\mu(A(x_k,x_{k'},n,m)>0$ for all $n\in\mathbb{N}^+$. In particular, we have $\mu(A(x_k,x_{k'},M,m)>0$. Hereafter, we fix such $k,k',m\in\mathbb{N}^+$ and for simplicity we put $B=A(x_k,x_{k'},M,m)$.  
\medskip

$\textit{Lusin Theorem}$ \cite{F} implies that for every $\epsilon>0$ there is a measurable subset $C_\epsilon$ with $\mu(X\setminus C_\epsilon)<\epsilon$ such that $f^i\mid_{C_\epsilon}$ is continuous for all integer $i$ with $\mid i\mid \leq m$. Taking $\epsilon=\frac{\mu(B)}{2}$, we obtain a measurable subset $C=C_{\frac{\mu(B)}{2}}$ such that $f^i\mid_C$ continuous for all integer $i$ with $\mid i\mid\leq m$ and $\mu(B\cap C)>0$. Since $B\cap C$ is Lindel{\"o}f, by Lemma \ref{3.5} there are $z\in B\cap C$ and $\delta_0>0$ such that $\mu(B\cap C\cap B[z,\delta])>0$ for all $0<\delta<\delta_0$. Since $z\in C$ and $f^i\mid_C$ is continuous for all $\mid i\mid\leq m$, we can fix $0<\delta<\delta_0$ such that $d(f^i(z),f^i(w))\leq\delta_z$ for all $\mid i\mid\leq m$, whenever $d(z,w)\leq\delta$ with $w\in C$. 
\medskip

Let $w\in B\cap C\cap B[z,\delta]$ which implies $w\in C\cap B[z,\delta]$ and hence, $d(f^i(z),f^i(w))\leq \delta_z$ for all $\mid i\mid\leq m$. Further, since $z,w\in B\cap C$, we have $d(f^i(z),f^i(w))\leq \frac{1}{M}\leq\delta_z$ for all $\mid i\mid\geq m$. Combining we get $d(f^i(z),f^i(w))\leq\delta_z$ for all $i\in\mathbb{Z}$ which implies $w\in\Gamma_{\delta_z}(z)$ and hence, $B\cap C\cap B[z,\delta]\subset\Gamma_{\delta_z}(z)$. Thus, $\mu(B\cap C\cap B[z,\delta])=0$, which is a contradiction.  
\end{proof}

\begin{Corollary} (i) A uniform equivalence $f$ of a separable metric space, is aperiodic with respect to any pointwise expansive measure $\mu$ for $f$.
\\ 
(ii) The set $A(f)$ of a pointwise measure expansive homeomorphism $f$ of a separable metric space, has measure zero with respect to any non-atomic Borel measure.
\\ 
(iii) Let $f$ be a uniform equivalence of a separable metric space. Then, the set of periodic points of $f$ has measure zero with respect to any pointwise expansive measure $\mu$ for $f$.
\\  
(iv) The set $Per(f)$ of a pointwise measure expansive uniform equivalence $f$ of a separable metric space is at most countable.     
\label{3.7}
\end{Corollary}
\begin{proof}
(i) If $B$ is a set and $n$ be a positive integer such that $f^n(x)=x$ for all $x\in B$, then $B\subset A(f^m)$ for some $1\leq m\leq n$. By Corollary \ref{3.2}, $\mu$ is pointwise expansive with respect to $f^m$. Then by Theorem \ref{3.6}, we have $\mu(A)=0$. Since $A$ and $n$ are arbitrary, we conclude that $f$ is aperiodic.
\\
(ii) This clearly follows from the definitions of pointwise expansive measure and measure expansive homeomorphism. 
\\
(iii) It is clear from the fact that $Per(f)=\bigcup_{m\in\mathbb{N}^{+}} Fix(f^m)$, $Fix(f^m)\subset A(f^m)$ and $\mu$ is pointwise expansive for $f^m$ for all $m\geq 1$.     
\\
(iv) It is enough to show that $Fix(f^n)$ is at most countable for each $n\in\mathbb{N}^{+}$. If possible suppose that $Fix(f^n)$ is uncountable for some $n\in\mathbb{N}^{+}$. Then, $Fix(f^n)$ is uncountable, separable, complete metric space and hence there is non-atomic Borel measure $\nu$ on $Fix(f^n)$. If we set $\mu(A)=\nu(Fix(f^n)\cap A)$ for all Borel measurable set $A$ of $X$, then $\mu$ is a non-atomic Borel measure on $X$ such that $Fix(f^n)$ has full $\mu$-measure which implies $Per(f)$ has full $\mu$-measure. On contrary, $\mu(Per(f))=0$ by Corollary \ref{3.7}(iii) and hence, $Per(f)$ must be at most countable.      
\end{proof}

\begin{theorem}
The set $A_f(p,q)$ under a bi-measurable map $f$ of a separable metric space $X$, has measure zero with respect to any pointwise expansive outer regular measure.    
\label{3.11}
\end{theorem}
\begin{proof}
Let $\mu$ be a pointwise expansive outer regular measure for $f$. Let $p,q\in X$ be given and $A\subset A(p,q)$ be such that $\mu(A)>0$. For $n\in\mathbb{N}^+$, let $A(n)$ be the set of points $a\in A$ such that $\frac{1}{n}\leq\delta_a$. Since $A=\bigcup_{n\in\mathbb{N}^+} A(n)$, we have $\mu(A(M))>0$ for some $M\in\mathbb{N}^+$. If $A_N=\lbrace x\in X\mid d(f^n(x),f^n(p))\leq\frac{1}{2M}$ for all $n\geq N$ and $d(f^n(x),f^n(q))\leq\frac{1}{2M}$ for all $n\leq -N\rbrace$, then $A(M)\subset\bigcup_{N\geq 0}{A_N}$ and each $A_N$ is measurable. We show that $\mu (\bigcup_{N\geq 1}A_N)=0$. If possible, suppose $\mu(\bigcup_{N\geq 1} A_N)>0$. So, there is $K\geq 1$ such that $\mu(A_K)>0$.  
\medskip

$\textit{Lusin Theorem}$ \cite{F} implies that for every $\epsilon>0$ there is a measurable subset $C_\epsilon$ with $\mu(X\setminus C_\epsilon)<\epsilon$ such that $f^i\mid_{C_\epsilon}$ is continuous for all integer $i$ with $\mid i\mid \leq K$. Taking $\epsilon=\frac{\mu(A_K)}{2}$, we obtain a measurable subset $C=C_{\frac{\mu(A_K)}{2}}$ such that $f^i\mid_C$ continuous for all integer $i$ with $\mid i\mid\leq K$ and $\mu(A_K\cap C)>0$. Since $A_K\cap C$ is Lindel{\"o}f, by Lemma \ref{3.5} there are $z\in B\cap C$ and $\delta_0>0$ such that $\mu(A_K\cap C\cap B[z,\delta])>0$ for all $0<\delta<\delta_0$. Since $z\in C$ and $f^i\mid_C$ is continuous for all $\mid i\mid\leq K$, we can fix $0<\delta<\delta_0$ such that $d(f^i(z),f^i(w))\leq\delta_z$ for all $\mid i\mid\leq K$, whenever $d(z,w)\leq\delta$ with $w\in C$.
\medskip

Let $w\in A_K\cap C\cap B[z,\delta]$ which implies $w\in C\cap B[z,\delta]$ and hence, $d(f^i(z),f^i(w))\leq \delta_z$ for all $\mid i\mid\leq K$. Since $z\in A_K\cap C$, we have $w,z\in A_K$. Hence, $d(f^i(w),f^i(p))\leq\frac{1}{2M}$, $d(f^i(z),f^i(q))\leq\frac{1}{2M}$ for all $i\geq K$ and $d(f^i(w),f^i(p))\leq\frac{1}{2M}$, $d(f^i(z),f^i(q))\leq\frac{1}{2M}$ for all $i\leq -K$. Thus, $d(f^i(z),f^i(w))\leq\frac{1}{M}\leq\delta_z$ for all $\mid i \mid \geq K$. Thus $d(f^i(z),f^i(w))\leq\delta_z$ for all $i\in\mathbb{Z}$ which implies $A_K\cap C\cap B[z,\delta])\subset \Gamma_{\delta_z}(z)$ and hence, $\mu(A_K\cap C\cap B[z,\delta]))=0$, which is a contradiction.
\end{proof}

\begin{Corollary}
(i) The set $A_f(p,q)$ of a pointwise measure expansive homeomorphism of a separable metric space, has measure zero with respect to any non-atomic Borel measure.
(ii) The set of all heteroclinic points of a pointwise measure expansive uniform equivalence of a separable metric space has measure zero with respect to non-atomic Borel measure. 
\label{3.13}
\end{Corollary}

\begin{theorem}
Every stable class of a measurable map on a separable metric space has measure zero with respect to any positively pointwise expansive outer regular measure.
\label{3.15}
\end{theorem}
\begin{proof}
One can prove this result by following the same steps as in Theorem \ref{3.11} above.
\end{proof}

\begin{Corollary}
A continuous map of a separable metric space admitting positively pointwise expansive outer regular measure has uncountably many stable classes. 
\label{3.16}
\end{Corollary}

\begin{theorem}
If a bi-measurable map $f$ with canonical coordinates admits strictly positive positively pointwise expansive measure $\mu$, then no point is a sink. 
\label{3.17}
\end{theorem}
\begin{proof}
If $x\in X$ is a sink, then there is $\delta_0>0$ such that $W^u(x,\delta_0)=\lbrace x\rbrace$. Let $\delta_x>0$ be a pointwise expansive constant at $x$ and let $\epsilon=$ min$\lbrace \delta_0,\delta_x\rbrace$. Since $f$ has canonical coordinates, there is $\delta'>0$ such that $d(x,y)<\delta'$ implies $W^s(x,\epsilon)\cap W^u(y,\epsilon)\neq\phi$. Observe that $W^u(x,\epsilon)\subset W^u(x,\delta_0)=\lbrace x\rbrace$ and hence, $W^u(x,\epsilon)=\lbrace x\rbrace$. If $y\in X$ be such that $d(x,y)<\delta'$, then because of symmetric property of metric, $W^s(y,\epsilon)\cap W^u(x,\epsilon)\neq\phi$. But $W^u(x,\epsilon)=\lbrace x\rbrace$, we have $x\in W^s(y,\epsilon)$ which implies $y\in W^s(x,\epsilon)$ and $\epsilon\leq \delta_x$, we have $y\in W^s(x,\delta_x)$ proving $B_{\delta_x}(x)\subset W^s(x,\delta_x)$. But then $\mu(B_{\delta_x}(x))\leq\mu(W^s(x,\delta_x))=0$ which is a contradiction to the fact that $\mu$ is strictly positive.     
\end{proof}


\begin{Example}
Let $g$ be a pointwise expansive homeomorphism on an uncountable compact metric space $(Y,d_0)$. Let $p$ be a periodic point of $g$ with prime period $t$. Let $X=Y\cup E$, where $E$ is an infinite enumerable set. So, there is a bijection $r:\mathbb{N}\rightarrow E$. Consider $Q=\bigcup_{k\in\mathbb{N}} \lbrace 1,2,3\rbrace\times\lbrace k\rbrace\times\lbrace 0,1,2,3,...,t-1\rbrace$. Suppose $s:Q\rightarrow \mathbb{N}$ is a bijection. Then, consider the bijection $q:Q\rightarrow E$ given by $q(i,k,j)=r(s(i,k,j))$. Thus, any point $x\in E$ has the form $x=q(i,k,j)$ for some $(i,k,j)\in Q$. 
\medskip

Define a function $d:X\times X\rightarrow\mathbb{R}^+$ by  
\[d(a,b)=\begin{cases} 
0 & \textnormal {if $a=b$},
\\
d_0(a,b) & \textnormal {if $a,b\in Y$} 
\\
\frac{1}{k}+d_0(g^j(p),b) & \textnormal {if $a=q(i,k,j)$ and $b\in Y$}
\\
\frac{1}{k}+d_0(a,g^j(p)) & \textnormal {if $a\in Y$ and $b=q(i,k,j)$} 
\\
\frac{1}{k} & \textnormal {if $a=q(i,k,j)$,$b=q(l,k,j)$ and $i\neq l$} 
\\
\frac{1}{k}+\frac{1}{m}+d_0(g^j(p),g^r(p)) & \textnormal {if $a=q(i,k,j)$,$b=q(i,m,r)$ and $k\neq m$ or $j\neq r$}    
\end{cases}\] 
Then similarly as in \textit{Proposition 3.1} and \textit{Proposition 3.2} \cite{CC}, one can show that $(X,d)$ is a compact metric space.  
\pagebreak

Define a map $f:X\rightarrow X$ by 
\begin{center}
\[f(x)=\begin{cases}
g(x) & \textnormal {if $x\in Y$}
\\
q(i,k,(j+1))$ mod $t & \textnormal {if $x=q(i,k,j)$} 
\end{cases}\] 
\end{center}

The similar procedure as in \cite{CC} follows to prove that $f$ is a homeomorphism. Since for any $\delta>0$ there is $K\in\mathbb{N}^+$ such that $\frac{1}{k}<\delta$ for all $k\geq K$, $\Gamma_{\frac{1}{K}}(p)$ contains countably many points from $E$. Therefore, $f$ is not pointwise expansive. Since $X$ is non-atomic, $f$ is pointwise measure-expansive. Here, $f$ can not be strongly pointwise expansive because of the following theorem.       
\label{3.19}
\end{Example}

\begin{theorem}
If $f$ is a strongly pointwise measure-expansive map with at least one periodic point. Then, $f/_{Per(f)}$ is pointwise expansive.   
\label{3.20}
\end{theorem}
\begin{proof}
Let $p\in Per(f)$ and $\delta_p$ be the strong measure-expansive constant for $f$ at $p$. Further, suppose $q\in Per(f)$ such that $q\in\Gamma_{\delta_p}(p)$. Let $m$ and $n$ be the period of $p$ and $q$ respectively. Let $\mu$ be the measure such that for each $x\in O(p)\cup O(q)$, we get $\mu(x)=\frac{1}{m+n}$. Then, $\mu(\Gamma_{\delta_p}(p))\geq \mu(p)+\mu(q)=\frac{2}{m+n}> \mu(\lbrace p\rbrace)$. This leads to a contradiction. Thus, $f/_{Per(f)}$ must be pointwise expansive. 
\end{proof}

The following discussion distinguishes the notion of expansivity and pointwise expansivity.    
\medskip

\begin{theorem}
If $f$ is pointwise $N$-expansive, then it is pointwise expansive. In particular, $N$-expansive homeomorphisms are pointwise expansive.  
\label{3.21}
\end{theorem}
\begin{proof}
Let $x\in X$ and $\delta_x$ be the pointwise $N$-expansive constant for $f$ at $x$. Let $y_1,...,y_k\in\Gamma_{\delta_x}(x)$, where $k\leq N$. If $\epsilon_x=$min$_{1\leq i\leq k} d(x,y_i)$, then $\Gamma_{\epsilon_x}(x)=\lbrace x\rbrace$. Since $x$ is arbitrary, $f$ is pointwise expansive with pointwise expansive constant $\epsilon_x$ at $x$.
\end{proof}

\begin{Corollary}
There is no pointwise $N$-expansive homeomorphism of an arc, circle or $2$-cell. 
\label{3.22}
\end{Corollary}
\begin{proof}
This follows from Theorem \ref{3.21} and a \textit{Corollary }of \cite{R1970}. 
\end{proof}

\begin{Corollary}
Let $X$ be a compact metric space and let $f$ be a pointwise $N$-expansive homeomorphism. If $X$ is an infinite minimal set for $f$, then $f$ cannot be distal.
\label{3.23}
\end{Corollary}
\begin{proof}
This follows from Theorem \ref{3.21} and a \textit{Corollary} of \cite{R1970}.
\end{proof}


\begin{Example}
If the metric space is non-atomic, then strong pointwise measure-expansive systems are pointwise measure-expansive. Now, consider $X$ to be the subspace of $\mathbb{R}^n$ ($n\geq 1$) given by $\lbrace a,b\rbrace\cup\lbrace x_i\mid i\in\mathbb{Z}\rbrace$, where $a\neq b$, $x_i\neq x_j$ for $i\neq j$ and $\lbrace x_i\rbrace_{i\in\mathbb{Z}}$ is a bi-infinite sequence such that lim$_{i\to\infty} x_i=a$ and lim$_{i\to -\infty} x_i=b$. Then, the identity map $f$ on $X$ is not pointwise expansive and hence, by Theorem \ref{3.20} it is not strongly pointwise measure-expansive. But if $Y=X\setminus\lbrace a,b\rbrace$, then the identity map on $Y$ is pointwise expansive and strongly pointwise measure-expansive. 
\label{3.25}
\end{Example}

\section{Pointwise Specification} 
In this section, we consider the dynamics of continuous map on any metric space (unless otherwise stated). In particular, we study the effect of the presence of specification points on the nature of the system under different hypothesis.   

\begin{Definition}
Let $f$ be a continuous map on $X$ and let $x\in X$. Then, $x$ is said to be a specification (resp. periodic specification) point of $f$ if for every $\epsilon > 0$ there exists a positive integer $M(\epsilon)$ such that for any finite sequence $x = x_{1}, x_{2}, ..., x_{k}$ in $X$ and any integers $0\leq a_{1}\leq b_{1} < a_{2}\leq b_{2} < . . .< a_{k}\leq b_{k}$ with $a_{j} - b_{j-1} \geq M(\epsilon)$ for $1\leq j \leq k$, there exists $y\in X$ (resp. $y\in Per(f)$) such that $d(f^{i}(y), f^{i}(x_{j})) < \epsilon$ for $a_{j}\leq i\leq b_{j}$, $1\leq j\leq k$. A map $f$ is said to have pointwise specification (resp. pointwise periodic specification) property if every point is a specification (resp. periodic specification) point.  
\end{Definition} 
\medskip

If $f$ has specification (resp. periodic specification) property, then every point is a specification (resp. periodic specification) point. Further, every specification point of $f$ is a specification point of $f^{m}$ for any $m > 0$.  
\medskip

A continuous onto map $\pi : \oversl{X}\rightarrow X$ is said to be locally isometric covering map if for each $x\in X$, there exists an open set $U(x)$ containing $x$ such that $\pi^{-1}(U(x)) = \cup_{\alpha}U_{\alpha}$, where $\lbrace U_{\alpha}\rbrace$ is a pairwise disjoint family of open sets such that $\pi|_{U_{\alpha}} : U_{\alpha}\rightarrow U(x)$ is an isometry for each $\alpha$. 
\medskip

\begin{theorem}
Let $f$ and $g$ be continuous maps on $\overline{X}$ and $X$ respectively. Let $\pi: \oversl{X}\rightarrow X$ be a locally isometric covering map. Further, suppose that $\pi f = g\pi$ and there exists $\delta_{0} >0$ such that for each $\oversl{x}\in \oversl{X}$ and $0<\delta \leq\delta_{0}$, $\pi : U_{\delta}(\oversl{x})\rightarrow U_{\delta}(\pi(\oversl{x}))$ is an isometry. Then, the following statements are true  
\begin{enumerate}
\item[(i)] If $x$ is a specification point of $g$ then every point in $\pi^{-1}(x)$ is a specification point of $f$.
\item[(ii)] If $\oversl{x}$ is a specification point of $f$ then $\pi(\oversl{x})$ is a specification point of $g$.
\end{enumerate}
\end{theorem}
\begin{proof}
(i) Suppose that $x$ is a specification point of $g$. For any $0< \epsilon\leq \delta_{0}$, choose an integer $M(\epsilon)$ due to the specification of $g$ at $x$. Let $\oversl{x}\in \pi^{-1}(x)$. Consider a sequence of points $\oversl{x} = \oversl{x}_{1}, \oversl{x}_{2}, . . ., \oversl{x}_{k}$ and $0\leq a_{1}\leq b_{1} < a_{2}\leq b_{2}<. . .< a_{k}\leq b_{k}$ with $a_{j} - b_{j-1} \geq M(\epsilon)$ for all $1\leq j \leq k$. Choose $\lbrace x_{i} : 1\leq i\leq k\rbrace$ such that $x_{i} = \pi(\oversl{x}_{i})$ for all $1\leq i\leq k$. Choose $y\in X$ such that $d(g^{i}(y), g^{i}(x_{j})) < \epsilon \leq \delta_{0}$ for all $a_{j}\leq i\leq b_{j}$ and all $1\leq j\leq k$ or $d(g^{i}(\pi(\oversl{y})), g^{i}(\pi(\oversl{x}_{j}))) < \epsilon \leq \delta_{0}$ for all $a_{j}\leq i\leq b_{j}$ and all $1\leq j\leq k$ or $d(\pi (f^{i}(\oversl{y})), \pi(f^{i}(\oversl{x}_{j}))) < \epsilon \leq \delta_{0}$ for all $a_{j}\leq i\leq b_{j}$ and all $1\leq j\leq k$, for some $\oversl{y}\in \oversl{X}$. Since $\pi$ is an isometry on $U_{\delta}(f^{i}(\oversl{y}))$ for all $a_{j}\leq i\leq b_{j}$ and all $1\leq j\leq k$, we get that $\oversl{d}(f^{i}(\oversl{y}), f^{i}(\oversl{x}_{j})) < \epsilon \leq \delta_{0}$ for all $a_{j}\leq i\leq b_{j}$ and all $1\leq j\leq k$. 
\medskip

(ii) Suppose that $\oversl{x}$ is a specification point of $f$. For any $0< \epsilon\leq \delta_{0}$, choose an integer $M(\epsilon)$ due to the specification of $\oversl{x}$. Let $x = \pi(\oversl{x})$. Consider a sequence of points $x = x_{1}, x_{2}, . . ., x_{k}$ and $0\leq a_{1}\leq b_{1} < a_{2}\leq b_{2}<. . .< a_{k}\leq b_{k}$ with $a_{j} - b_{j-1} \geq M(\epsilon)$ for all $1\leq j \leq k$. Choose $\lbrace \oversl{x}_{i} : 1\leq i\leq k\rbrace$ such that $x_{i} = \pi(\oversl{x}_{i})$ and $\oversl{x}_{1} = \oversl{x}$ for all $1\leq i\leq k$. Choose $\oversl{y}\in \oversl{X}$ such that $d(f^{i}(\oversl{y}), f^{i}(\oversl{x}_{j})) < \epsilon \leq \delta_{0}$ for all $a_{j}\leq i\leq b_{j}$ and all $1\leq j\leq k$. 
Since $\pi$ is an isometry on $U_{\epsilon}(\oversl{y})$, $\oversl{d}(f^{i}(\oversl{y}), f^{i}(\oversl{x}_{j})) =  d(\pi f^{i}(\oversl{y}), \pi f^{i}(\oversl{x}_{j})) = d(g^{i}(\pi(\oversl{y})), g^{i} \pi (\oversl{x}_{j})) < \epsilon \leq \delta_{0}$ or $d(g^{i}(y), g^{i} (x_{j})) < \epsilon \leq \delta_{0}$ for $y =\pi(\oversl{y})$, for all $a_{j}\leq i\leq b_{j}$ and all $1\leq j\leq k$. Hence $x$ is a specification point of $g$. 
\end{proof} 

\begin{Definition} Let $f$ be a continuous map on $X$ and let $x\in X$. Then, 
\begin{enumerate} 
\item [(i)] $x$ is said to be a sensitive point of $f$ if there exists $\delta_{x} > 0$ such that for every open set $U$ containing $x$ there exists $y\in U$ such that $d(f^{n}(x), f^{n}(y))>\delta_{x}$ for some $n\in \mathbb{N}^{+}$. Such a constant $\delta_{x}$ is said be a sensitivity constant for $f$ at $x$. A map $f$ is said to be pointwise sensitive if every point is a sensitive point of $f$.
\item [(ii)] $x$ is said to be a topologically transitive (topologically mixing) point of $f$ if for any open set $U$ containing $x$ and any non-empty open set $V$ there exists $n\in \mathbb{N}^{+}$ ($N\in \mathbb{N}^{+}$) such that $f^{n}(U)\cap V \neq \phi$ (for all $n\geq N$). 
\item [(iii)] $f$ is said to have a dense set of periodic points at $x$ if every deleted open set of $x$ contains a periodic point of $f$. 
\end{enumerate}   
\end{Definition}

It is easy to verify that every $N$-expansive map on a metric space without isolated points is sensitive and hence, it is pointwise sensitive. But we don't know whether pointwise sensitivity implies sensitivity. Further, observe that $f$ is topologically mixing (resp. topologically transitive) if and only if every point is topologically mixing (resp. topologically transitive) point of $f$. Moreover, every topologically mixing point is a topologically transitive point. 

\begin{Definition}
Let $f$ be a continuous map on $X$ and let $x\in X$. Then, $x$ is said to be a Devaney chaotic point of $f$ if the following holds
\begin{enumerate}
\item[(i)] $x$ is topologically transitive of $f$;
\item[(ii)] $f$ has dense set of periodic points at $x$;
\item[(iii)] $x$ is sensitive point of $f$. 
\end{enumerate}

A continuous map $f$ is said to be pointwise Devaney chaotic if every point in $X$ is Devaney chaotic.  
\label{4.7}
\end{Definition}

\begin{theorem}
Every specification point of a continuous map $f$ on $X$ is a topologically mixing point. 
\label{4.8}
\end{theorem}
\begin{proof}
Let $f$ be a continuous map on $X$ and let $x\in X$ be a specification point of $f$. Assume that $U$ is any open set containing $x$ and $V$ is any non-empty open set in $X$. Fix $y\in V$ and $\epsilon > 0$ such that $B_{\epsilon}(x) \subset U$ and $B_{\epsilon}(y) \subset V$. Choose a positive integer $M(\epsilon)$ correspond to the specification point $x$ of $f$. Thus for any $n\geq M(\epsilon)$, set $a_{1} = 0 = b_{1}$, $a_{2} = n = b_{2}$, $x_{1} = x$, $x_{2} = f^{-n}(y)$. Choose $z_{n}\in X$ such that $d(f^{i}(z_{n}), f^{i}(x_{j})) < \epsilon$ for all $a_{j}\leq i\leq b_{j}$ and all $1\leq j\leq 2$. Hence, $z_{n}\in B_{\epsilon}(x)$ and $f^{n}(z_{n})\in B_{\epsilon}(y)$. Since this is true for all $n\geq M(\epsilon)$ and any open set $U$ containing, we get that $x$ is a topologically mixing point of $f$. 
\end{proof}

\begin{Corollary}
If a continuous map $f$ on $X$ has pointwise specification property, then $f$ is topologically mixing.
\label{4.9}
\end{Corollary}

\begin{theorem}
If $x$ is a periodic specification point of a continuous surjective map $f$ on $X$, then $f$ has dense set of periodic points at $x$.   
\end{theorem}
\begin{proof}
Let $U$ be any deleted open set of $x$ and choose $\epsilon > 0$ such that $B_{\epsilon}(x)\setminus\lbrace x\rbrace \subset U$. Choose positive integer $M=M(\epsilon)$ correspond to the periodic specification point $x$ of $f$. Let $a_{1} = 0 = b_{1}$, $a_{2} = M = b_{2}$ and consider $x_{1}$, $x_{2}$, where $x_1=x$ and $x_{2}\in X$ is arbitrary. Then, there exists a periodic point $z\in X$ satisfying $d(f^{i}(z), f^{i}(x_{j})) < \epsilon$ for $a_{j}\leq i\leq b_{j}$, $1\leq j\leq 2$. If $x$ is non-periodic point, then $z\in B_{\epsilon}(x)\setminus \lbrace x\rbrace\subset U$. If $x$ is periodic, then choose $x_2$ such that $d(f^M(x_2),\mathcal{O}(x))>\epsilon$. Then, $z\neq x$ and hence, $z\in U$. Since $U$ was chosen arbitrary, $x$ has dense set of periodic points at $x$.  
\end{proof}

\begin{theorem}
Let $f$ be a continuous map on $X$. If $x\in X$ is topologically transitive point such that $f$ has dense set of periodic points at $x$, then $x$ is a sensitive point of $f$.
\end{theorem}
\begin{proof}
Since $f$ has dense set of periodic points at $x$, we can choose $\delta > 0$ and a periodic point $q$ such that $d(x, \mathcal{O}(q)) \geq \frac{\delta}{2}$. We claim that $\eta = \frac{\delta}{8}$ is a sensitivity constant for $f$ at $x$. It is clear that $d(x, \mathcal{O}(q)) \geq \frac{\delta}{2} = 4\eta$. Let $N$ be an open set containing $x$ and choose a periodic point $p$ in $N'\setminus \lbrace x\rbrace$, where $N'=(N\cap B_{\eta}(x))$. Suppose that $p$ has prime period $n$ and that $W = \cap_{i = 0}^{n}f^{-i}(B_{\eta}(f^{i}(q)))$. Clearly, $W$ is a non empty open set in $X$. Since $x$ is a topologically transitive point, there exists $y\in N'$ and $k\in \mathbb{N}^{+}$ such that $f^{k}(y)\in W$. 
Let $j$ be the integer part of $(\frac{k}{n} + 1)$. Since $0\leq nj-k\leq n$, we have 
\[f^{nj}(y) = f^{nj-k+k}(y) = f^{nj-k}(f^{k}(y))\in f^{nj-k}(W) \subset B_{\eta}(f^{nj-k}(q))\] 
By the triangle inequality,  
\begin{align*}
d(x, f^{nj-k}(q)) &\leq d(x, p) + d(p, f^{nj-k}(q))\\ 
&\leq d(x, p) + d(p, f^{nj}(y)) + d(f^{nj}(y), f^{nj-k}(q))
\end{align*}
or
\begin{align*}
d(f^{nj}(p), f^{nj}(y)) &= d(p, f^{nj}(y)) \\
&\geq d(x, f^{nj-k}(q)) - d(x, p) - d(f^{nj}(y), f^{nj-k}(q))
\end{align*}
Since $p\in N'\setminus \lbrace x\rbrace$, $f^{nj}(y)\in B_{\eta}(f^{nj-k}(q))$, $d(x, \mathcal{O}(q)) > 4\eta$ and $f^{nj-k}\in \mathcal{O}(q)$, we have $-d(x, p) > -\eta$, $-d(f^{nj}(y), f^{nj-k}(q)) > -\eta$ and $d(x, f^{nj-k}(q)) > 4\eta$. Therefore, 
\begin{align*}
2\eta &= 4\eta - \eta - \eta  \\
&< d(x, f^{nj-k}(q)) - d(x, p) - d(f^{nj}(y), f^{nj-k}(q))\\
&< d(f^{nj}(p), f^{nj}(y)) \\
&\leq d(f^{nj}(p), f^{nj}(x)) + d(f^{nj}(x), f^{nj}(y))
\end{align*}
The last inequality implies that $d(f^{nj}(p), f^{nj}(x)) > \eta$ or $d(f^{nj}(x), f^{nj}(y)) > \eta$. Hence, $\eta$ is a sensitivity constant for $f$ at $x$.
\end{proof}

\begin{Corollary}
Every continuous surjective map with pointwise periodic specification property is Devaney chaotic.  
\end{Corollary}

Recall that shadowable points for homeomorphisms on compact metric space is defined by Morales in \cite{MSP}, in which it has been proved that every homeomorphism on compact metric space has shadowing if and only if every point is a shadowable point. By Remark 1.1 \cite{MSP}, the definition can extended to continuous map as follows. 
\medskip 

A sequence $\lbrace x_{i}\rbrace_{i\in \mathbb{N}}$ is said to be through a subset $B$ of $X$ if $x_{0}\in X$. A point $x\in X$ is said to be shadowable point for $f$ if for every $\epsilon > 0$ there exists a $\delta > 0$ such that every $\delta$-pseudo orbit through $x$ can be $\epsilon$-traced. Following similar steps as in Theorem 1.1 \cite{MSP}, we can prove the following.  
\medskip

\begin{theorem} If $f$ is continuous onto map on compact metric space, then $f$ has shadowing if and only if every point is a shadowable point.
\end{theorem}  

\begin{theorem}
If $f$ is a topologically mixing continuous map on a totally bounded metric space $X$, then every shadowable point is a specification point. 
\label{4.13}
\end{theorem}
\begin{proof}
Let $\epsilon > 0$ and choose $0 <\delta < \frac{\epsilon}{2}$ such that every $\delta$-pseudo orbit through $x$ can be $\frac{\epsilon}{2}$-traced. Since $X$ is totally bounded, there exists $\lbrace z_{1}, z_{2}, . . .,z_{m}\rbrace \subset X$ such that $X = \cup_{i=1}^{m}B_{\frac{\delta}{2}}(z_{i})$. Since $f$ is topologically mixing, there exists $M \in \mathbb{N}^{+}$ such that $f^{n}(B_{\frac{\delta}{2}}(z_{i})) \cap B_{\frac{\delta}{2}}(z_{j}) \neq \phi$ for all $n\geq M$ and all $1\leq i, j\leq m$. 
Choose $0\leq a_{1}\leq b_{1} < a_{2}\leq b_{2} < . . . < a_{k}\leq b_{k}$, with $a_{j} - b_{j-1}\geq M$ for all $2\leq j\leq k$ and $x = x_{1}, x_{2}, . . ., x_{k}$.	Let $B(z)$ denotes a set in $\lbrace B_{\frac{\delta}{2}}(z_{i})\rbrace$ containing $z$. Since $a_{j} - b_{j-1} \geq M$ for all $2\leq j \leq k$, by topologically mixing of $f$ we can choose $y_{j}\in B(f^{b_{j}}(x_{j}))$ such that $f^{a_{j}-b_{j-1}}(y_{j})\in B(f^{a_{j+1}}(x_{j+1}))$. We define a sequence 
\[   
v_{i} = 
     \begin{cases}
       f^{i}(x) & 0\leq i < a_{1}\\
       f^{i}(x_{j}) & a_{j}\leq i < b_{j} \\
       f^{i - b_{j}}(y_{j}) & b_{j} \leq i < a_{j+1} \\ 
       f^{i}(x_{k}) & i \geq b_{k} \\ 
     \end{cases}
\]

Clearly, $\lbrace v_{i}\rbrace_{i\in \mathbb{N}}$ forms a $\delta$-pseudo orbit of $f$ through $x$. Hence there exists $w\in X$ such that $d(f^{i}(w), v_{i}) < \frac{\epsilon}{2}$ for all $i\in \mathbb{N}$,  which implies that $d(f^{i}(w), f^{i}(x_{j})) < \epsilon$ for all $a_{j}\leq i\leq b_{j}$ and all $1\leq j\leq k$. Hence $x$ is a specification point.
\end{proof}

\begin{Corollary}
For every uniformly continuous surjective map with the shadowing property on a totally bounded metric space, pointwise specification property implies specification property. 
\end{Corollary}
\begin{proof}
Proof follows from the Main Theorem of \cite{PT}, Corollary \ref{4.9} and Theorem \ref{4.13}. 
\end{proof}

\begin{theorem}
Let $f$ be a uniformly continuous surjective map on $X$. If $f$ has two distinct specification points, then $f$ has positive Bowen entropy.   
\label{4.15}
\end{theorem}
\begin{proof}
Let $x \neq y$ be two distinct specification points of $f$. Let $\epsilon > 0$ be such that $d(x, y) > 3\epsilon$. Further, let $M(\epsilon) = M$ be a positive integer given by the specification points $x$ and $y$ of $f$. Consider two distinct $(n+1)$-tuples, $(z_{0}, z_{1}, . . ., z_{n})$ and $(z'_{0}, z'_{1}, . . ., z'_{n})$ with $z_{0} = x$, $z'_{0}= y$, $z_{i}, z'_{i}\in \lbrace x, y\rbrace$ for all $1\leq i\leq n$. Choose $a_{0} = 0 = b_{0}$, $a_{1} = M = b_{1}, . . ., a_{n} = nM = b_{n}$. Then, there exist $z ,z'\in X$ corresponding to the specification points $x$ and $y$ respectively. Observe that $z \neq z'$, otherwise $d(f^{i}(z), f^{i}(z_{i})) < \epsilon$ and $d(f^{i}(z), f^{i}(z'_{i})) < \epsilon$ for all $a_{j}\leq i\leq b_{j}$, $0\leq j\leq n$. In particular, $d(x, y)< 2\epsilon$, a contradiction. 
Now suppose that $z_{0} = z'_{0}$. Since $f$ is onto, fix $z_{n+1}\in f^{-(n+1)M}(x)$, $z'_{n+1}\in f^{-(n+1)M}(y)$. Choose $a_{n+1} = (n+1)M = b_{n+1}$, $(z_{0}, z_{1}, . . ., z_{n}, z_{n+1})$ and $(z'_{0}, z'_{1}, . . ., z'_{n}, z'_{n+1})$. Let $z, z'\in X$ such that $d(f^{i}(z), f^{i}(z_{i})) < \epsilon$ and $d(f^{i}(z'), f^{i}(z'_{i})) < \epsilon$ for all $a_{j}\leq i\leq b_{j}$, $0\leq j\leq (n+1)$. If $z = z'$ then $d(x, y) = d(f^{(n+1)M}(z_{n+1}), f^{(n+1)M}(z'_{n+1})) \leq d(f^{(n+1)M}(z_{n+1}), f^{(n+1)M}(z)) + d(f^{(n+1)M}(z), f^{(n+1)M}(z'_{n+1})) < 2\epsilon$, a  contradiction. Hence for distinct $2^{n+1}$-tuples, we can choose $2^{n+1}$-distinct tracing points. 
Similarly, if $d(f^{i}(z), f^{i}(z')) < \epsilon$ for all $0\leq i\leq (n+1)M$, where $z$ and $z'$ are tracing points of above sequences then $d(x, y) < 3\epsilon$, a contradiction. 
Hence there exists atleast $2^{n+1}$ points which are $((n+1)M, \epsilon)$-separated. Hence,

\begin{align*}
h(f, X) &= sup_{K\in \mathcal{K(X)}} \lim_{\epsilon\rightarrow 0} s(\epsilon, K) \\
&\geq \lim_{\epsilon\rightarrow 0} s(\epsilon, K) \\
& \geq \limsup_{n\rightarrow \infty}\frac{1}{n}log s_{n}(\epsilon, K) \\
&\geq \limsup_{n\rightarrow \infty} \frac{1}{(n+1)M}log(2^{n+1}) \\ 
&= \frac{log2}{M} > 0
\end{align*}
\end{proof}

\begin{Corollary}
Every topologically mixing uniformly continuous surjective map having atleast two distinct shadowable points on a totally bounded metric space has positive Bowen entropy.
\end{Corollary}

\begin{Example}
Let $f : [0, \infty) \rightarrow [0, \infty)$ be defined by $f(x)=2x$. Note that $f$ has positive Bowen entropy and has shadowing property. It is easy to see that, $y=0$ is the only topologically mixing point and hence, $f$ can have at most one specification point. Thus, two distinct points is not necessary for a uniformly continuous surjective map to have positive Bowen entropy. In fact, one can observe that $f$ has no specification point. Hence, a point which is both topologically mixing point and shadowable point need not be a specification point.  
\end{Example}

\begin{Example}
Let $f : [0, 1] \rightarrow [0, 1]$ be defined as $f(x) = x^{2}$. Note that $f$ has zero Bowen entropy and has shadowing property. It is easy to see that $y = 1$ is the only topologically mixing point. In fact $y=1$ is a shadowable point but not a specification point. Thus, topologically mixing and shadowable point on a totally bounded metric space need not be a specification point. 
\end{Example}
\medskip

Acknowledgements: First author is supported by Department of Science and Technology, Government of India, under INSPIRE Fellowship (Resgistration No.- IF150210) Program. Second author is supported by CSIR-Junior Research Fellowship (File No.-09/045(1558)/2018-EMR-I) of Government of India.

\end{document}